\documentclass[graybox,envcountsame]{svmult}
\usepackage{mathptmx}       % selects Times Roman as basic font
\usepackage{helvet}         % selects Helvetica as sans-serif font
\usepackage{courier}        % selects Courier as typewriter font
\usepackage[bottom]{footmisc}% places footnotes at page bottom

\usepackage{amsmath,amssymb}

\smartqed

\makeatletter
\newcommand{\ls}{\leqslant}
\newcommand{\gr}{\geqslant}
\makeatother

\newcommand{\norm}[1]{\left\Vert#1\right\Vert}

\newcommand{\set}[1]{\left\{#1\right\}}
\newcommand{\brac}[1]{\left(#1\right)}
\newcommand{\scalar}[1]{\left \langle #1 \right \rangle}

\newcommand{\Real}{\mathbb{R}}

\newcommand{\vrad}{{\rm vrad}}
\newcommand{\Eps}{\mathcal{E}}

\begin{document}

\title*{$M$-estimates for isotropic convex bodies and their $L_q$-centroid bodies}
% Use \titlerunning{Short Title} for an abbreviated version of
% your contribution title if the original one is too long
\author{Apostolos Giannopoulos and Emanuel Milman}
% Use \authorrunning{Short Title} for an abbreviated version of
% your contribution title if the original one is too long
\institute{Apostolos Giannopoulos \at  Department of Mathematics, University of Athens, Panepistimioupolis 157-84, Athens, Greece.
\email{apgiannop@math.uoa.gr}
\and Emanuel Milman \at Department of Mathematics,
Technion - Israel Institute of Technology, Haifa 32000, Israel. \email{emilman@tx.technion.ac.il}}
%
% Use the package "url.sty" to avoid
% problems with special characters
% used in your e-mail or web address
%

\maketitle

\begin{abstract}
\footnotesize Let $K$ be a centrally-symmetric convex body in $\Real^n$ and let $\|\cdot\|$ be its induced norm on
${\mathbb R}^n$. We show that if $K \supseteq r B_2^n$ then:
\[
\sqrt{n} M(K) \ls C \sum_{k=1}^{n}  \frac{1}{\sqrt{k}} \min\left(\frac{1}{r} , \frac{n}{k} \log\Big(e + \frac{n}{k}\Big) \frac{1}{v_{k}^{-}(K)}\right) .
\]
where $M(K)=\int_{S^{n-1}} \|x\|\, d\sigma(x)$ is the mean-norm, $C>0$ is a universal constant, and $v^{-}_k(K)$ denotes the minimal volume-radius of a $k$-dimensional orthogonal projection of $K$.
We apply this result to the study of the mean-norm of an isotropic convex body $K$ in ${\mathbb R}^n$ and its $L_q$-centroid bodies.
In particular, we show that if $K$ has isotropic constant $L_K$ then:
\[
 M(K)\ls \frac{C\log^{2/5}(e+ n)}{\sqrt[10]{n}L_K}  .
\]
\end{abstract}

\section{Introduction}

Let $K$ be a centrally-symmetric convex compact set with non-empty interior (``body") in Euclidean space $({\mathbb R}^n,\langle\cdot ,\cdot\rangle)$.
We write $\|\cdot\|$ for the norm induced on ${\mathbb R}^n$ by $K$ and $h_K$ for the support function of $K$; this is precisely the dual norm $\|\cdot\|^{\ast }$.
The parameters:
\begin{equation}\label{eq:1.2}
M(K)=\int_{S^{n-1}} \|x\|\, d\sigma(x)\quad\hbox{and}\quad
M^{\ast }(K)=\int_{S^{n-1}}h_K(x)\,d\sigma (x),
\end{equation}
where $\sigma$ denotes the rotationally invariant probability measure
on the unit Euclidean sphere $S^{n-1}$, play a central role in the asymptotic theory of finite
dimensional normed spaces.

Let $\vrad(K) := \brac{|K|/|B_2^n|}^{1/n}$ denote the volume-radius of $K$, where $|A|$ denotes Lebesgue measure in the linear hull of $A$ and $B_2^n$ denotes the unit Euclidean ball. It is easy to check that:
\begin{equation}
M(K)^{-1}\ls \vrad(K) \ls M^{\ast }(K)  = M(K^\circ) ,
\end{equation}
where $K^\circ = \set{ y \in \Real^n \,:\, \scalar{x,y} \ls 1 \;\hbox{for all}\; x \in K}$ is the polar body to $K$, i.e. the unit-ball of the dual norm $\|\cdot\|^\ast$. Indeed,
the left-hand side is a simple consequence of Jensen's inequality after
we express the volume of $K$ as an integral in polar coordinates, while the right-hand side is
the classical Urysohn inequality. In particular, one always has $M(K)M^{\ast }(K)\gr 1$.

In the other direction, it is known from results of Figiel--Tomczak-Jaegermann \cite{Figiel-Tomczak-1979}, Lewis \cite{Lewis-1979} and Pisier's
estimate \cite{Pisier-1982} on the norm of the Rademacher projection, that for any centrally-symmetric convex body $K$, there exists $T\in GL(n)$ such that:
\begin{equation}\label{eq:MM*}
M(TK) M^{\ast }(TK)\ls C \log n  ,
\end{equation}
where $C >0$ is a universal constant. Throughout this note, unless otherwise stated, all constants $c,c',C,\ldots$ denote universal numeric constants, independent of any other parameter, whose value may change from one occurrence to the next. We write $A\simeq B$ if  there exist absolute constants $c_1,c_2>0$ such that $c_1 A\ls B\ls c_2 A$.

The role of the linear map $T$ in (\ref{eq:MM*}) is to put the body in a good ``position", since without it $M(K) M^*(K)$ can be arbitrarily large.
The purpose of this note is to obtain good upper bounds on the parameter $M(K)$, when $K$ is already assumed to be in a good position - the isotropic position.
A convex body $K$ in ${\mathbb R}^n$ is called isotropic if it has
volume $1$, its barycenter is at the origin, and there exists a
constant $L_K>0$ such that:
\begin{equation}\label{eq:1.1}
\int_K \langle x,\theta\rangle^2 dx=L_K^2 ,~ \hbox{for all}\; \theta \in S^{n-1} .
\end{equation}
It is not hard to check that every convex body $K$ has an isotropic affine image which is uniquely determined
up to orthogonal transformations \cite{Milman-Pajor-LK}. Consequently, the isotropic constant $L_K$ is an affine invariant of $K$. A
central question in asymptotic convex geometry going back to Bourgain \cite{BourgainMaximalFunctionsOnConvexBodies} asks if there exists
an absolute constant $C>0$ such that $L_K\ls C$ for every (isotropic)
convex body $K$ in $\Real^n$ and every $n \gr 1$. Bourgain \cite{Bourgain-1991} proved that $L_K\ls
C\sqrt[4]{n}\log n$ for every centrally-symmetric convex body $K$ in ${\mathbb
R}^n$. The currently best-known general estimate, $L_K\ls
C\sqrt[4]{n}$, is due to Klartag \cite{Klartag-2006} (see also the work of Klartag and E. Milman
\cite{Klartag-EMilman-2012} and a further refinement of their approach by Vritsiou \cite{Vritsiou-2013}).

It is known that if $K$ is a centrally-symmetric isotropic convex body in ${\mathbb R}^n$ then
$K\supseteq L_K B_2^n$, and hence trivially $M(K)\ls 1 / L_K$.
It seems that, until recently, the problem of bounding $M(K)$ in isotropic position had not been studied and there were
no other estimates besides the trivial one. The example of the normalized $\ell_\infty^n$ ball shows that the
best one could hope is $M(K)\ls C\sqrt{\log n}/\sqrt{n}$. Note that obtaining a bound of the form $M(K)\ls n^{-\delta }L_K^{-1}$
immediately provides a non-trivial upper bound on $L_K$, since $M(K) \gr \vrad(K)^{-1} \simeq 1/\sqrt{n}$, and hence
$L_K\ls c^{-1} n^{\frac{1}{2}-\delta }$. The current best-known upper bound on $L_K$
suggests that $M(K) \ls C (n^{1/4} L_K)^{-1} $ might be a plausible goal.

Paouris and Valettas (unpublished) proved that
for every isotropic centrally-symmetric convex body $K$ in ${\mathbb R}^n$ one
has:
\begin{equation}\label{eq:6.5.1}
M(K)\ls \frac{C\sqrt[3]{\log (e+n)}}{\sqrt[12]{n}L_K} .
\end{equation}
Subsequently, this was extended by Giannopoulos, Stavrakakis, Tsolomitis and Vritsiou in \cite{Giannopoulos-Stavrakakis-Tsolomitis-Vritsiou-TAMS}
to the case of the $L_q$-centroid bodies $Z_q(\mu )$ of an isotropic
log-concave probability measure $\mu$ on $\mathbb R^n$ (see Section 5 for the necessary
definitions). The approach of \cite{Giannopoulos-Stavrakakis-Tsolomitis-Vritsiou-TAMS}
was based on a number of observations regarding the local structure of $Z_q(\mu)$;
more precisely, lower bounds for the in-radius of their proportional projections and estimates for their
dual covering numbers (we briefly sketch an improved version of this approach in Section 7).

In this work we present a different method, applicable to general centrally-symmetric convex
bodies, which yields better quantitative estimates. As always, our starting point is Dudley's entropy estimate (see e.g. \cite[Theorem 5.5]{Pisier-book}):
\begin{equation}\label{eq:Dudley}
\sqrt{n} M^*(K) \ls C \sum_{k \gr 1}\frac{1}{\sqrt{k}} e_k(K,B_2^n) ,
\end{equation}
where $e_k(K,B_2^n)$ are the entropy numbers of $K$. Recall that the covering number $N(K,L)$ is defined to be the minimal number of translates of $L$ whose union covers $K$, and that $e_k(K,L) := \inf \set{ t > 0 \, : \, N(K,t L) \ls 2^k}$.

Our results depend on the following natural volumetric parameters associated with $K$ for each $k=1,\ldots,n$:
\[
w_k(K) := \sup \set{ \vrad(K \cap E): E\in G_{n,k}} ~,~ v^{-}_k(K) := \inf \set{ \vrad(P_E(K)) : E \in G_{n,k} } ,
\]
where $G_{n,k}$ denotes the Grassmann manifold of all $k$-dimensional linear subspaces of $\Real^n$, and $P_E$ denotes orthogonal projection onto $E \in G_{n,k}$. Note that by the Blaschke--Sanatal\'o inequality and its reverse form due to Bourgain and V. Milman (see Section \ref{sec:prelim1}), it is immediate to verify that $w_k(K^\circ) \simeq \frac{1}{v^{-}_k(K)}$.

\begin{theorem} \label{thm:covering2-intro}
For every centrally-symmetric convex body $K$ in $\Real^n$ and $k \gr 1$:
\[
e_k(K,B_2^n) \ls C  \frac{n}{k} \log \Big(e+\frac{n}{k}\Big) \sup_{1 \ls m \ls \min(k,n)} \set{ 2^{-\frac{k}{3m}} w_m(K) } .
\]
\end{theorem}

By invoking Carl's theorem (see Section \ref{sec:prelim1}), a slightly weaker version of Theorem \ref{thm:covering2-intro} may be deduced from the following stronger statement:

\begin{theorem}\label{thm:main-s-intro}
Let $K$ be a centrally-symmetric convex body in $\Real^n$. Then for any $k=1,\ldots,\lfloor n/2 \rfloor$
there exists $F \in G_{n,n-2k}$ so that:
\begin{equation} \label{eq:F-estimate-intro}
K \cap F \subseteq C \frac{n}{k} \log\Big(e + \frac{n}{k}\Big) w_{k}(K) B_2^n \cap F  ,
\end{equation}
and dually, there exists $F \in G_{n,n-2k}$ so that:
\begin{equation} \label{eq:PF-estimate-intro}
P_F(K) \supseteq \frac{1}{C \frac{n}{k} \log(e + \frac{n}{k})} v^{-}_{k}(K) P_F(B_2^n) .
\end{equation}
\end{theorem}

A weaker version of Theorem \ref{thm:main-s-intro}, with the parameters $w_k(K)$, $v^{-}_k(K)$ above replaced by:
\[
v_k(K) := \sup \set{ \vrad(P_E(K)) : E \in G_{n,k} } ~,~ w^{-}_k(K) := \inf \set{ \vrad(K \cap E) : E\in G_{n,k}} ,
\]
respectively, was obtained by V. Milman and G. Pisier in \cite{VMilman-Pisier-1987}
(see Theorem \ref{thm:milman-pisier}). Our improved version is crucial for properly exploiting the corresponding properties of isotropic convex bodies.

By (essentially) inserting the estimates of Theorem \ref{thm:covering2-intro} into \eqref{eq:Dudley} (with $K$ replaced by $K^{\circ }$), we obtain that if $K$ is a centrally-symmetric convex body in $\Real^n$ with $K \supseteq r B_2^n$ then:
\begin{equation}\label{eq:general-M}
\sqrt{n} M(K) \ls C \sum_{k=1}^{n}  \frac{1}{\sqrt{k}} \min\left(\frac{1}{r} , \frac{n}{k} \log\Big(e + \frac{n}{k}\Big) \frac{1}{v_{k}^{-}(K)}\right) .
\end{equation}

In the case of the centroid bodies $Z_q(\mu)$ of an isotropic log-concave probability measure $\mu $ on ${\mathbb R}^n$,
one can obtain precise information on the growth of the parameters $v_k^{-}(Z_q(\mu ))$. We recall the relevant definitions in Section \ref{sec:prelim2}, and use \eqref{eq:general-M} to deduce in Section \ref{sec:results} that:
\begin{equation} \label{eq:MZq-intro}
2\ls q\ls q_0 := (n\log n)^{2/5} \;\;\; \Longrightarrow \;\;\; M(Z_q(\mu)) \ls C\frac{\sqrt{\log q}}{\sqrt[4]{q}} .
\end{equation}

In particular, since $Z_n(\mu) \supseteq Z_{q_0}(\mu)$ and $M(K) \simeq M(Z_n(\lambda_{K / L_K})) / L_K$, where $\lambda_A$ denotes the uniform probability measure on $A$, we immediately obtain:

\begin{theorem}\label{th:current-M-estimate}
If $K$ is a centrally-symmetric isotropic convex body in $\mathbb R^n$ then:
\begin{equation}\label{eq:current-M-estimate}
 M(K)\ls \frac{C \log^{2/5}(e+n)}{\sqrt[10]{n}L_K}.
\end{equation}
\end{theorem}

It is clear that \eqref{eq:current-M-estimate} is not optimal. Note that if (\ref{eq:MZq-intro}) were to remain valid until $q_0 = n$, we would obtain the bound $M(K) \ls C \frac{\sqrt{\log (e+n)}}{n^{1/4} L_K}$, which as previously explained would in turn imply that $L_K \ls C \sqrt{\log (e+n)} \; n^{1/4}$, in consistency with the best-known upper bound on the isotropic constant. We believe that it is an interesting question to extend the range where (\ref{eq:MZq-intro}) remains valid. In Section \ref{sec:results}, we obtain such an extension when $\mu$ is in addition assumed to be $\Psi_\alpha$ (see Section \ref{sec:results} for definitions).

\medskip

Our entire method is based on Pisier's regular versions of V. Milman's $M$-ellipsoids associated to a given centrally-symmetric convex body $K$, comparing between volumes of sections and projections of $K$ and those of its associated regular ellipsoids. This expands on an approach already employed in \cite{Pisier-book,BKM,Klartag-VMilman-2005,Klartag-Psi2,Giannopoulos-Stavrakakis-Tsolomitis-Vritsiou-TAMS}.

\medskip

We conclude the introduction by remarking that the dual question of providing an upper bound for the mean-width $M^{\ast }(K)$ of an
isotropic convex body $K$ has attracted more attention in recent years. Until recently, the best known estimate
was $M^{\ast }(K)\ls C n^{3/4}L_K$, where $C>0$ is an absolute
constant (see \cite[Chapter 9]{BGVV-book-isotropic} for a number of proofs of this inequality).
The second named author has recently obtained in \cite{EMilman-2014} an essentially optimal answer to this question -
for every isotropic convex body $K$ in ${\mathbb R}^n$ one has $M^{\ast }(K)\ls C\sqrt{n}\log^2n\,L_K$.

\section{Preliminaries and notation from the local theory} \label{sec:prelim1}

Let us introduce some further notation. Given $F \in G_{n,k}$,
we denote $B_F=B_2^n\cap F$ and $S_F=S^{n-1}\cap F$.
A centrally-symmetric convex body $K$ in ${\mathbb R}^n$ is a compact convex set with non-empty interior so that $K = -K$. The
norm induced by $K$ on ${\mathbb R}^n$ is given by $\| x\|_K=\min\{ t\gr 0:x\in tK\}$. The support function of $K$ is defined by $h_K(y):=\norm{y}_K^* =\max \bigl\{\langle y,x\rangle : x\in K\bigr\}$, with $K^\circ$ denoting the unit-ball of the dual-norm. By the Blaschke--Santal\'o inequality (the right-hand side below) and its reverse form due to Bourgain and V. Milman
\cite{Bourgain-VMilman-1987} (the left-hand side), it is known that:
\begin{equation} \label{eq:BS}
0 < c \ls \vrad(K)\vrad(K^{\circ }) \ls 1 .
\end{equation}
Recall that the $k$-th entropy number is defined as
$$e_k(K,L) := \inf \set{ t > 0 \; :\; N(K,t L) \ls 2^k }.$$
A deep and very useful fact about entropy numbers
is the Artstein--Milman--Szarek duality of entropy theorem \cite{Artstein-VMilman-Szarek-2004},
which states that:
\begin{equation}\label{eq:artstein-szarek-milman}
e_k(B_2^n, K) \ls C e_{ck}(K^\circ,B_2^n)
\end{equation}
for every centrally-symmetric convex body $K$ and $k\gr 1$.

In what follows, a crucial role is played by G. Pisier's
regular version of V. Milman's $M$-ellipsoids. It was shown by Pisier (see \cite{Pisier-1989}
or \cite[Chapter 7]{Pisier-book}) that for any centrally-symmetric convex body
$K$ in $\Real^n$ and $\alpha \in (0,2)$, there exists an ellipsoid $\Eps = \Eps_{K,\alpha}$ so that:
\begin{equation} \label{eq:Pisier-ek}
\max\{e_k(K,\Eps),e_k(K^\circ,\Eps^\circ),e_k(\Eps,K),e_k(\Eps^\circ,K^\circ)\} \ls P_\alpha \brac{\frac{n}{k}}^{1/\alpha},
 \end{equation}
 where $P_\alpha \ls C \left(\frac{\alpha}{2-\alpha}\right)^{1/2}$ is a positive constant depending only on $\alpha $.

 Given a pair of centrally-symmetric convex bodies $K,L$ in $\Real^n$, the Gelfand numbers $c_k(K,L)$ are defined as:
\[
c_k(K,L) := \begin{cases} \inf \set{ {\rm diam}_{L \cap F} (K \cap F) : F \in G_{n,n-k} } & k = 0,\ldots,n-1 \\ 0 & \text{otherwise} \end{cases} ,
\]
where ${\rm diam}_{A}(B) := \inf \set{ R > 0 : B \subseteq R A}$. We denote $c_k(K) = c_k(K,B_2^n)$ and $e_k(K) = e_k(K,B_2^n)$.

Carl's theorem \cite{Carl-1981} relates any reasonable Lorentz norm of the sequence of  entropy numbers $\set{e_m(K,L)}$ with that of the Gelfand numbers $\set{c_m(K,L)}$. In particular, for any $\alpha >0$, there exist constants $C_{\alpha },C'_{\alpha}>0$ such that for any $k \gr 1$:
\begin{equation}\label{eq:carl}
\sup_{m=1,\ldots,k} m^{\alpha} e_m(K,L) \ls C_\alpha \sup_{m=1,\ldots,k} m^{\alpha} c_m(K,L) ,
\end{equation}
and:
\begin{equation} \label{eq:carl2}
\sum_{m=1}^k m^{-1+\alpha} e_m(K,L) \ls C'_\alpha \sum_{m=1}^k  m^{-1+\alpha} c_m(K,L)  .
\end{equation}
In fact, Pisier deduces the covering estimates of \eqref{eq:Pisier-ek}
from an application of Carl's theorem, after establishing the following estimates:
\begin{equation} \label{eq:Pisier-ck}
\max\{c_k(K,\Eps) , c_k(K^\circ, \Eps^\circ)\} \ls P_\alpha \brac{\frac{n}{k}}^{1/\alpha}\;\;\hbox{for all}\; k \in \set{1,\ldots,n} .
\end{equation}

Our estimates depend on a number of volumetric parameters of $K$, already defined in the Introduction, which we now recall:
\begin{equation*}
w_k(K) := \sup \set{ \vrad(K \cap E) : E \in G_{n,k} } , v_k(K) := \sup \set{ \vrad(P_E (K)):E \in G_{n,k} },
\end{equation*}
and
\begin{equation*}
w^{-}_k(K) := \inf \set{ \vrad(K \cap E) : E \in G_{n,k} } , v^{-}_k(K) := \inf \set{ \vrad(P_E (K)) : E \in G_{n,k} }.
\end{equation*}
Note that $0 < c \ls w^{-}_k(K) v_k(K^\circ) , v^{-}_k(K) w_k(K^\circ) \ls 1$ by (\ref{eq:BS}). Also observe that $k \mapsto v_k(K)$ is non-increasing by
the Alexandrov inequalities and Kubota's formula, and that $k \mapsto w^{-}_k(K)$ is non-decreasing
by polar-integration and Jensen's inequality.

We refer to the books \cite{VMilman-Schechtman-book} and
\cite{Pisier-book} for additional basic facts from the local theory of
normed spaces.

\section{New covering estimates} \label{sec:covering}

The main result of this section provides a general upper bound for the entropy numbers $e_k(K,B_2^n)$.

\begin{theorem} \label{thm:covering1}
Let $K$ be a centrally-symmetric convex body in $\Real^n$, and let $k \gr 1$. Then:
\[
e_k(K,B_2^n) \ls C  \frac{n}{k} \log \Big(e+\frac{n}{k}\Big) \sup_{1 \ls m \ls \min(k,n)} \set{ 2^{-\frac{k}{3m}} w_m(K) }.
\]
\end{theorem}

We combine this fact with Dudley's entropy estimate
\begin{equation}\label{eq:dudley-pisier}
\sqrt{n} M^*(K) \ls C \sum_{k \gr 1}\frac{1}{\sqrt{k}} e_k(K,B_2^n).
\end{equation}
(see \cite[Theorem 5.5]{Pisier-book} for this formulation). As an immediate consequence, we obtain:

\begin{corollary}\label{cor:covering1}
Let $K$ be a centrally-symmetric convex body in $\Real^n$ with $K \subseteq R B_2^n$. Then:
\[
\sqrt{n} M^*(K) \ls C \sum_{k \gr 1} \frac{1}{\sqrt{k}} \min \set{R, \frac{n}{k} \log \Big(e+\frac{n}{k}\Big) \sup_{1 \ls m \ls \min(k,n)} \set{  2^{-\frac{k}{3m}} w_m(K)} }  .
\]
Dually, let $K$ be a centrally-symmetric convex body in $\Real^n$ with $K \supseteq r B_2^n$. Then:
\[
\sqrt{n} M(K) \ls C \sum_{k \gr 1}  \frac{1}{\sqrt{k}} \min \set{\frac{1}{r}, \frac{n}{k} \log \Big(e+\frac{n}{k}\Big) \sup_{1 \ls m \ls \min(k,n)} \set{  2^{-\frac{k}{3m}} \frac{1}{v_m^{-}(K)} } }
\]
\end{corollary}

\begin{proof}
The first claim follows by a direct application of \eqref{eq:dudley-pisier} if we estimate $e_k(K,B_2^n)$
using Theorem \ref{thm:covering1} and the observation that $e_k(K,B_2^n)\ls R$ for all $k\gr 1$ (recall that $K\subseteq RB_2^n$).
Then, the second claim follows by duality since $w_m(K^\circ) \simeq \frac{1}{v^{-}_m(K)}$.
\qed
\end{proof}

We will see in the next section that the supremum over $m$ above is unnecessary and that one may always use $m=k$,
only summing over $k=1,\ldots,n$. But we proceed with the proof of Theorem \ref{thm:covering1}, as it is a simpler approach.

\medskip

%\begin{proof}[of Theorem \ref{thm:covering1}]
\noindent {\it Proof of Theorem \ref{thm:covering1}}.
Assume without loss of generality that $k$ is divisible by 3, and use the estimate:
\[
e_k(K,B_2^n) \ls e_{k/3}(K,\Eps) e_{2k/3}(\Eps,B_2^n),
\]
where $\Eps = \Eps_{K,\alpha_k}$ is Pisier's $\alpha_k$-regular $M$-ellipsoid associated to $K$, with $\alpha_k \in [1,2)$ to be determined. The first term is controlled directly by Pisier's regular covering estimate (\ref{eq:Pisier-ek}).
For the second term we use the following simple fact about covering numbers of ellipsoids (see e.g. \cite[Remark 5.15]{Pisier-book}):
\[
e_{j}(\Eps,B_2^n) \simeq \sup_{1 \ls m \ls n} 2^{-j/m} w_m(\Eps) \simeq \sup_{1 \ls m \ls \min(j,n)} 2^{-j/m} w_m(\Eps);
\]
the latter equivalence follows since $w_m(\Eps)$ is the geometric average of the $m$ largest principal radii of $\Eps$, and so $m \mapsto w_m(\Eps)$ is non-increasing.
Now recall that
\begin{equation}\label{eq:step1}w_m(\Eps) \simeq 1/v_m^{-}(\Eps^{\circ }).\end{equation}
To estimate $v_m^{-}(\Eps^{\circ })$, we use a trivial volumetric bound: for any $E\in G_{n,m}$,
\begin{align*}\frac{\vrad(P_E (K^\circ))}{\vrad(P_E (\Eps^\circ)) e_s(K^\circ,\Eps^\circ)} &\ls N(P_E (K^\circ ), e_s(K^\circ,\Eps^\circ) P_E (\Eps^\circ))^{1/m}\\
&\ls N(K^\circ , e_s(K^\circ,E^\circ) \Eps^\circ)^{1/m} \ls 2^{s/m} ,
\end{align*}
for $s \gr 1$ to be determined. Consequently:
\[
v_m^{-}(\Eps^\circ) \gr \frac{1}{2^{s/m} e_s(K^\circ,\Eps^\circ)} v_m^{-}(K^\circ) ,
\]
and plugging this back into (\ref{eq:step1}), we deduce:
\[
w_m(\Eps) \ls C 2^{s/m} e_s(K^\circ,\Eps^\circ) w_m(K) ,
\]
and hence:
\[
e_{2k/3}(\Eps,B_2^n) \ls C \sup_{1 \ls m \ls \min(k,n)} 2^{\frac{s-2k/3}{m}} e_{s}(K^\circ,\Eps^\circ) w_m(K) .
\]
Setting $s = k/3$, we conclude that:
\[
e_{2k/3}(\Eps,B_2^n) \ls C e_{k/3}(K^\circ,\Eps^\circ)  \sup_{1 \ls m \ls \min(k,n)} 2^{-\frac{k}{3m}} w_m(K) .
\]
Combining everything, we obtain:
\begin{eqnarray*}
e_k(K , \Eps) & \ls & C e_{k/3}(K,\Eps) e_{k/3}(K^\circ,\Eps^\circ) \sup_{1 \ls m \ls \min(k,n)} 2^{-\frac{k}{3m}} w_m(K) \\
& \ls & \frac{C'}{2-\alpha_k} \brac{\frac{n}{k}}^{\frac{2}{\alpha_k}} \sup_{1 \ls m \ls \min(k,n)} 2^{-\frac{k}{3m}} w_m(K) .
\end{eqnarray*}
Setting $\alpha_k = 2 - \frac{1}{\log(e + n/k)}$, the assertion follows.
\qed
%\end{proof}

\begin{remark} \label{rem:duality} %\rm
Theorem \ref{thm:covering1} implies the following dual covering estimate:
\begin{equation}\label{eq:remark-3-3}
e_k(B_2^n,K) \ls C  \frac{n}{k} \log \left(e+\frac{n}{k}\right) \sup_{1 \ls m \ls \min(k,n)} \set{ 2^{-\frac{k}{3m}} \frac{1}{v_m^{-}(K)} } .
\end{equation}
Indeed, this is immediate from the duality of entropy theorem \eqref{eq:artstein-szarek-milman}
and the fact that $w_m(K^\circ) \simeq \frac{1}{v_m^{-}(K)}$. Alternatively, one may simply repeat the proof of Theorem \ref{thm:covering1} with the roles of $K$ and $B_2^n$ exchanged.
\end{remark}

\section{New diameter estimates}\label{section-9-3} \label{sec:diam}

This section may be read independently of the rest of this work, and contains a refinement of the following result of V. Milman and G. Pisier from \cite{VMilman-Pisier-1987}, as exposed in \cite[Lemma 9.2]{Pisier-book}:

\begin{theorem}[Milman--Pisier]\label{thm:milman-pisier}
Let $K$ be a centrally-symmetric convex body in $\Real^n$. Then, for any $k=1,\ldots,n/2$:
\[
c_{2k}(K) \ls C \frac{n}{k} \log\Big(e + \frac{n}{k}\Big) v_{k}(K) .
\]
In other words, there exists $F \in G_{n,n-2k}$ so that:
\begin{equation} \label{eq:MP-1}
K \cap F \subseteq C \frac{n}{k} \log\Big(e + \frac{n}{k}\Big) v_{k}(K) B_F ,
\end{equation}
and dually, there exists $F \in G_{n,n-2k}$ so that:
\begin{equation} \label{eq:MP-2}
P_F(K) \supseteq \frac{1}{C \frac{n}{k} \log(e + \frac{n}{k})} w^{-}_{k}(K) B_F .
\end{equation}
\end{theorem}

Our version refines these estimates by replacing $v_{k}(K)$ and $w^{-}_k(K)$ above by the stronger $w_k(K)$ and $v^{-}_k(K)$ parameters, respectively; this refinement is crucial for our application in this paper.

\begin{theorem} \label{thm:main-s}
Let $K$ be a centrally-symmetric convex body in $\Real^n$. Then for any $k=1,\ldots,n/2$:
\[
c_{2k}(K) \ls C \frac{n}{k} \log\Big(e + \frac{n}{k}\Big) w_{k}(K) .
\]
In other words, there exists $F \in G_{n,n-2k}$ so that:
\begin{equation} \label{eq:F-estimate}
K \cap F \subseteq C \frac{n}{k} \log\Big(e + \frac{n}{k}\Big) w_{k}(K) B_F ,
\end{equation}
and dually, there exists $F \in G_{n,n-2k}$ so that:
\begin{equation} \label{eq:PF-estimate}
P_F(K) \supseteq \frac{1}{C \frac{n}{k} \log(e + \frac{n}{k})} v^{-}_{k}(K) B_F .
\end{equation}
\end{theorem}

Our refinement will come from exploiting the full strength of Pisier's result on the existence of regular $M$-ellipsoids. In contrast, the Milman--Pisier result is based on V. Milman's quotient-of-subspace theorem, from which it seems harder to obtain enough regularity to deduce our proposed refinement.

\medskip

%\begin{proof}
\noindent {\it Proof of Theorem \ref{thm:main-s}}.
Given $k=1,\ldots,n/2$, let $\Eps = \Eps_{K,\alpha_k}$ denote Pisier's $\alpha_k$-regular $M$-ellipsoid, for some $\alpha_k \in [1,2)$ to be determined. By the second estimate in (\ref{eq:Pisier-ck}), we know that there exists $E \in G_{n,n-k}$ so that:
\[
P_E(K) \supseteq \frac{1}{P_{\alpha_k}} \brac{\frac{k}{n}}^{1/\alpha_k} P_E (\Eps ).
\]
For the ellipsoid $\Eps' := P_E (\Eps )\subseteq E$, we may always find a linear subspace $F \subseteq E$ of codimension $m$ in $E$ so that:
\[
P_F (\Eps') \supseteq \inf_{H \in G_m(E)} \sup_{H' \subseteq H} \set{ \vrad(P_{H'} (\Eps')) } B_F,
\]
where $G_m(E)$ is the Grassmannian of all $m$-dimensional linear subspaces of $E$.
Indeed, this is immediate by choosing $H$ to be the subspace spanned by the $m$ shortest axes of $\Eps'$, and setting $F$ to be its orthogonal complement. Consequently, there exists a subspace $F \in G_{n,n-(k+m)}$ so that:
\begin{equation} \label{eq:basic}
P_F(K) \supseteq \frac{1}{P_{\alpha_k}} \brac{\frac{k}{n}}^{1/\alpha_k} \inf_{H \in G_{n,m}} \sup_{H' \subseteq H} \set{ \vrad(P_{H'} (\Eps)) } B_F .
\end{equation}

We now deviate from the proof of our refined version, to show how one may recover the Milman--Pisier estimate; the reader solely interested in the proof of our refinement may safely skip this paragraph. Assume for simplicity that $k < n/3$. By the first estimate in (\ref{eq:Pisier-ck}), we know that there exists $J \in G_{n,n-k}$ so that:
\[
K \cap J \subseteq P_{\alpha_k} \brac{\frac{n}{k}}^{1/\alpha_k} \Eps \cap J .
\]
Given $H \in G_{n,m}$ and denoting $H' := H \cap J \in G_{m'}(H)$ with $m' \in [m-k,m]$, it follows that:
\[
P_{H'}(\Eps) \supseteq \Eps \cap H' \supseteq  \frac{1}{P_{\alpha_k}} \brac{\frac{k}{n}}^{1/\alpha_k}  K \cap H' .
\]
Setting $m = 2k$, it follows from (\ref{eq:basic}) that there exists $F \in G_{n,n-3k}$ so that:
\[
P_F(K) \supseteq \frac{1}{P_{\alpha_k}^2} \brac{\frac{k}{n}}^{2/\alpha_k} \inf \set{ \vrad(K \cap H') \,:\, H' \in G_{n,m'} \;,\; m' \in [k,2k] } B_F .
\]
Noting that the sequence $m' \mapsto w_{m'}^{-}(K)$ is non-decreasing, and setting $\alpha_k = 2 - \frac{1}{\log(e + n/k)}$, we
have found $F \in G_{n,n-3k}$ such that
\[
P_F(K) \supseteq \frac{c}{\frac{n}{k} \log(e + \frac{n}{k})} w_k^{-}(K) ,
\]
as asserted in (\ref{eq:MP-2}) (with perhaps an immaterial constant $3$ instead of $2$). The assertion of (\ref{eq:MP-1}) follows by duality.

To obtain our refinement, we will use instead of the first estimate in (\ref{eq:Pisier-ck}), the covering estimate (\ref{eq:Pisier-ek}) (which Pisier obtains from (\ref{eq:Pisier-ck}) by an application of Carl's theorem, requiring the entire sequence of $c_k$ estimates, not just the one for our specific $k$). Setting $m=k$, we use a trivial volumetric estimate to control $\vrad(P_H (\Eps))$, exactly as in the proof of Theorem \ref{thm:covering1}: for any
$H\in G_{n,k}$,
\[
\frac{\vrad(P_H (K))}{\vrad(P_H (\Eps)) e_k(K,\Eps)} \ls N(P_H(K) , e_k(K,\Eps) P_H (\Eps))^{1/k} \ls N(K , e_k(K,E) \Eps)^{1/k} \ls 2 .
\]
Together with (\ref{eq:Pisier-ek}), we obtain:
\[
\vrad(P_H(\Eps)) \gr \frac{1}{2 e_k(K,\Eps)} \vrad(P_H (K)) \gr \frac{1}{2 P_{\alpha_k}} \brac{\frac{k}{n}}^{1/\alpha_k} \vrad(P_H (K)) .
\]
Plugging this into (\ref{eq:basic}) and setting as usual $\alpha_k = 2 - \frac{1}{\log(e + n/k)}$, the asserted estimate (\ref{eq:PF-estimate}) follows. The other estimate (\ref{eq:F-estimate}) follows by duality.
\qed
%\end{proof}

\medskip

As immediate corollaries, we have:

\begin{corollary} \label{cor:covering2}
For every centrally-symmetric convex body $K$ in $\Real^n$, $k=1,\ldots,n$ and $\alpha > 0$:
\[
e_k(K,B_2^n) \ls C_\alpha \sup_{m=1,\ldots,k}  \brac{\frac{m}{k}}^{\alpha} \frac{n}{m} \log\Big(e + \frac{n}{m}\Big) w_{m}(K),
\]
where $C_{\alpha }>0$ is a constant depending only on $\alpha $.
\end{corollary}

\begin{proof}
This is immediate from Theorem \ref{thm:main-s} and Carl's theorem \eqref{eq:carl}.
Note that $k \mapsto c_k(K,B_2^n)$ is non-increasing, and so there is no difference whether we take the supremum on the right-hand-side just on the even integers.
\qed
\end{proof}

\begin{corollary} \label{cor:M}
For every centrally-symmetric convex body $K$ in $\Real^n$ so that $K \subseteq R B_2^n$, we have:
\[
\sqrt{n} M^*(K) \ls C \sum_{k=1}^{n}  \frac{1}{\sqrt{k}} \min\left(R,\frac{n}{k} \log\Big(e + \frac{n}{k}\Big) w_{k}(K)\right) .
\]
Dually, for every centrally-symmetric convex body $K$ in $\Real^n$ so that $K \supseteq r B_2^n$, we have:
\[
\sqrt{n} M(K) \ls C \sum_{k=1}^{n}  \frac{1}{\sqrt{k}} \min\left(\frac{1}{r} , \frac{n}{k} \log\Big(e + \frac{n}{k}\Big) \frac{1}{v_{k}^{-}(K)}\right) .
\]
\end{corollary}

\begin{proof}
Let us verify the first claim, the second follows by duality. Indeed, this is immediate from Dudley's entropy estimate (\ref{eq:Dudley}) coupled with Carl's theorem (\ref{eq:carl2}):
\[
\sqrt{n} M^*(K) \ls C \sum_{k=1}^n \frac{1}{\sqrt{k}} e_k(K) \ls C' \sum_{k=1}^n \frac{1}{\sqrt{k}} c_k(K) .
\]
Obviously $c_k(K) \ls R$ for all $k$, and so the assertion follows from the estimates of Theorem \ref{thm:main-s}.
\qed
\end{proof}

Both Corollaries should be compared with the results of the previous section.

\begin{remark} \label{rem:alternative}\rm
It may be insightful to compare Theorem \ref{thm:main-s} to some other known estimates on diameters of $k$-codimensional sections, besides the Milman--Pisier Theorem \ref{thm:milman-pisier}. One sharp estimate is the Pajor--Tomczak-Jaegermann refinement \cite{PajorTomczakLowMStar} of V. Milman's low-$M^*$ estimate \cite{MilmanGeometricalInequalities}:
\begin{equation} \label{eq:low-M*}
c_k(L) \leq C \sqrt{\frac{n}{k}} M^*(L) ,
\end{equation}
for any origin-symmetric convex $L$ and $k = 1,\ldots,n$. However, for our application, we cannot use this to control $c_k(K^\circ)$ since we do not a-priori know $M^*(K^\circ) = M(K)$. A type of dual low-$M$ estimate was observed by Klartag \cite{Klartag-LowM}:
\[
c_k(L) \leq C^{\frac{n}{k}} \vrad(L)^{\frac{n}{k}} M(L)^{\frac{n-k}{k}} .
\]
Since $M(K^\circ) = M^*(K)$ is now well understood for an isotropic origin-symmetric convex body \cite{EMilman-2014}, this would give good estimates for low-dimensional sections (large codimension $k$), but unfortunately this is not enough for controlling $M(K)$. Klartag obtains the latter estimate from the following one, which is more in the spirit of the estimates we obtain in this work:
\[
c_k(L) \leq C^{\frac{n}{k}} \frac{\vrad(L)^{\frac{n}{k}}}{w_{n-k}(L)^{\frac{n-k}{k}}} .
\]
Again, this seems too rough for controlling the diameter of high-dimensional sections.
\end{remark}

\section{Preliminaries from asymptotic convex geometry}  \label{sec:prelim2}

An absolutely continuous Borel probability measure $\mu$ on  $\mathbb R^n$ is called
$\log$-concave if its density $f_\mu$ is of the form $\exp(-\varphi )$ with
$\varphi : \Real^n \rightarrow \Real \cup \set{+\infty}$ convex. Note that the uniform probability measure on $K$, denoted $\lambda_K$, is log-concave for any convex body $K$.

The barycenter of $\mu$ is denoted by ${\rm bar}(\mu ) :=\int_{\mathbb R^n}  x d\mu(x)$. The isotropic constant of $\mu$, denoted $L_\mu$, is the following affine invariant quantity:
\begin{equation}\label{definition-isotropic}
L_{\mu }:=(\sup_{x\in {\mathbb R}^n} f_{\mu} (x))^{\frac{1}{n}}  \det {\rm Cov}(\mu)^{\frac{1}{2n}},
\end{equation}
where ${\rm Cov}(\mu) := \int x \otimes x d\mu(x) - \int x d\mu(x) \otimes \int x d\mu(x)$ denotes the covariance matrix of $\mu$. We say
that a $\log $-concave probability measure $\mu $ on ${\mathbb R}^n$
is isotropic if ${\rm bar}(\mu )=0$ and ${\rm Cov}(\mu )$ is the identity matrix.
Note that a convex body $K$ of volume $1$ is isotropic if and only if the log-concave probability measure $\lambda_{K / L_K}$ is isotropic, and that $L_{\lambda_K}$ indeed coincides with $L_K$.
It was shown by K. Ball \cite{Ball-PhD,Ball-kdim-sections} that given $n \gr 1$:
\[
\sup_\mu {L_\mu} \ls C \sup_{K} L_K ,
\]
where the suprema are taken over all log-concave probability
measures $\mu$ and convex bodies $K$ in $\Real^n$, respectively (see
e.g. \cite{Klartag-2006} for the non-even case). Klartag's bound on the isotropic constant \cite{Klartag-2006} thus reads $L_\mu \ls C n^{1/4}$ for all
log-concave probability measures $\mu$ on $\Real^n$.

Given $E\in G_{n,k}$, we denote by $\pi_E\mu:= \mu \circ P_E^{-1}$
the push-forward of $\mu$ via $P_E$.
Obviously, if $\mu$ is centered or isotropic then so is $\pi_E \mu$,
and by the Pr\'{e}kopa--Leindler theorem, the same also holds for log-concavity.

Given a log-concave probability measure $\mu $ on ${\mathbb R}^n$ and $q\gr 1$,
the $L_q$-centroid body of $\mu$, denoted $Z_q(\mu )$, is the centrally-symmetric convex body with support
function:
\begin{equation}
h_{Z_q(\mu )}(y):= \left(\int_{{\mathbb R}^n} |\langle x,y\rangle|^{q}d\mu (x) \right)^{1/q}.
\end{equation}
Observe that $\mu $ is isotropic if and only if it is centered and
$Z_2(\mu )=B_2^n$. By Jensen's inequality
$Z_1(\mu )\subseteq Z_p(\mu )\subseteq Z_q(\mu )$ for all $1\ls p\ls
q<\infty $. Conversely, it follows from work of Berwald \cite{BerwaldMomentComparison} or by employing Borell's lemma (see \cite[Appendix III]{VMilman-Schechtman-book}), that:
\[ 1 \ls p \ls q \;\;\; \Longrightarrow \;\;\; Z_q(\mu )\subseteq C \frac{q}{p}Z_p(\mu ) .
\] When $\mu = \lambda_K$ is the uniform probability measure on a centrally-symmetric convex body $K$ in $\Real^n$, it is easy to check (e.g. \cite{BGVV-book-isotropic}) using the Brunn--Minkowski inequality that:
\[
c K \subseteq Z_n(\lambda_K) \subseteq K .
\]

Let $\mu$ denote an isotropic log-concave probability measure $\mu$ on $\Real^n$. It was shown by Paouris \cite{PaourisGAFA} that
\begin{equation}\label{eq:wZq-small}
1 \ls q \ls \sqrt{n} \;\;\; \Longrightarrow \;\;\; M^{\ast }\bigl(Z_q(\mu)\bigr)\simeq \sqrt{q} ,
\end{equation}
and that:
\begin{equation}
1 \ls q \ls n \;\;\; \Longrightarrow \;\;\; {\rm vrad}(Z_q(\mu)) \ls C \sqrt{q} .
\end{equation}
Conversely, it was shown by Klartag and E. Milman in \cite{Klartag-EMilman-2012} that:
\begin{equation}\label{eq:low-volume-Zq}
1 \ls q \ls \sqrt{n} \;\;\; \Longrightarrow \;\;\; {\rm vrad}(Z_q(\mu))\gr c_1\sqrt{q} .
\end{equation}
This determines the volume radius of $Z_q(\mu )$ for all $1 \ls q\ls\sqrt{n}$. For larger values of $q$ one can still use the lower bound:
\begin{equation}\label{eq:3}
1 \ls q \ls n \;\;\; \Longrightarrow \;\;\;  {\rm vrad}(Z_q(\mu)) \gr c_2\sqrt{q}\, L_{\mu}^{-1} ,
\end{equation}
obtained by Lutwak, Yang and Zhang \cite{Lutwak-Yang-Zhang-2000} via symmetrization.

\medskip

We refer to the book \cite{BGVV-book-isotropic} for further information on isotropic convex bodies and log-concave measures.

\section{$M$-estimates for isotropic convex bodies and their $L_q$-centroid bodies}  \label{sec:results-for-Zp} \label{sec:results}

Let $\mu$ denote an isotropic log-concave probability measure on $\Real^n$, and fix $H \in G_{n,k}$. A very useful observation is that:
\begin{equation*}
P_H\bigl(Z_q(\mu )\bigr) = Z_q\bigl(\pi_H(\mu )\bigr) .
\end{equation*}
It follows from \eqref{eq:low-volume-Zq} that:
\begin{equation}\label{eq:A2}
1 \ls q \ls \sqrt{k} \;\;\; \Longrightarrow \;\;\;  \vrad (P_H (Z_q(\mu ))) \gr c \sqrt{q} .
\end{equation}
Furthermore, using \eqref{eq:3}, we see that:
\begin{equation} \label{eq:A4}
q \ge \sqrt{k} \;\;\; \Longrightarrow \;\;\;  \vrad (P_H (Z_q(\mu ))) \gr c' \max \brac{\sqrt[4]{k},\frac{\sqrt{\min(q,k)}}{L_{\pi_H \mu}}} .
\end{equation}
Unfortunately, we can only say in general that $\sup\{L_{\pi_H \mu }:H \in G_{n,k}\}\ls C\sqrt[4]{k}$, and so the estimate (\ref{eq:A4}) is not very useful, unless we have some additional information on $\mu$. Recalling the definition of $v_k^{-}(Z_q(\mu ))$,  we summarize this (somewhat sloppily) in:

\begin{lemma}\label{lem:vk-Zq}
Let $\mu $ be an isotropic log-concave probability measure on ${\mathbb R}^n$. For any $q\gr 1$ and $k=1,\ldots,n$ we have:
\[
v_k^{-}(Z_q(\mu)) \gr c \sqrt{\min(q, \sqrt{k})}.
\]
Assuming that $\sup\{L_{\pi_H \mu }:H \in G_{n,k}\}\ls A_k$ we have:
\[
v_k^{-}(Z_q(\mu)) \gr \frac{c'}{A_k} \sqrt{\min(q, k)} .
\]
\end{lemma}

\subsection{Estimates for $Z_q(\mu)$}

Plugging these lower bounds for $v_k^{-}(Z_q(\mu))$ into either Theorem \ref{thm:covering1} or Corollary \ref{cor:covering2}
coupled with Remark \ref{rem:duality}, we immediately obtain estimates on the entropy numbers $e_k(B_2^n,Z_q(\mu ))$. Similar estimates on the maximal
(with respect to $F\in G_{n,n-k}$) in-radius of $P_F(Z_q(\mu))$
are obtained by invoking Theorem \ref{thm:main-s}.

\begin{theorem}\label{thm:Rqk-variant}
Given $q \gr 2$ and an integer $k = 1,\ldots,n$, denote:
\[
R_{k,q} := \min\set{ 1, C \frac{1}{\min(\sqrt{q} , \sqrt[4]{k})} \frac{n}{k} \log\left( e+ \frac{n}{k}\right) } .
\]
Then, for any isotropic log-concave probability measure $\mu$ on $\Real^n$:
\[
e_k(B_2^n,Z_q(\mu)) \ls R_{k,q} ,
\]
and there exists $F \in G_{n,n-k}$ so that:
\[
P_F (Z_q(\mu)) \supseteq \frac{1}{R_{k,q}} B_F .
\]
\end{theorem}

\begin{proof}
From \eqref{eq:remark-3-3} and Lemma \ref{lem:vk-Zq} we have:
\[
e_k(B_2^n,Z_q(\mu )) \ls C  \frac{n}{k} \log \left(e+\frac{n}{k}\right) \sup_{1 \ls m \ls k} \set{ 2^{-\frac{k}{3m}} \frac{1}{\min (\sqrt{q},\sqrt[4]{m})} } .
\]
Then, it suffices to observe that:
\begin{align*}\sup_{1 \ls m \ls k} \set{ 2^{-\frac{k}{3m}} \frac{1}{\min (\sqrt{q},\sqrt[4]{m})}}
&\simeq \sup_{1 \ls m \ls k} \set{ 2^{-\frac{k}{3m}} \left (\frac{1}{\sqrt{q}}+\frac{1}{\sqrt[4]{m}}\right )}\\
&\ls C\left (\frac{1}{\sqrt{q}}+\frac{1}{\sqrt[4]{k}}\right )\simeq \frac{1}{\min (\sqrt{q},\sqrt[4]{k})},
\end{align*}
because $2^{-\frac{k}{3m}}/\sqrt{q}\ls 1/\sqrt{q}$ for all $1\ls m\ls k$, and $m\mapsto 2^{\frac{k}{3m}}\sqrt[4]{m}$ attains its minimum at $m \simeq k$, 
so that $\sup_{1\ls m\ls k}(2^{-\frac{k}{3m}}/\sqrt[4]{m})\ls C/\sqrt[4]{k}$.
We also use the fact that in a certain range of values for $q \gr 2$ and $k \gr 1$, we might as well use the trivial estimates:
\begin{equation} \label{eq:trivial}
e_k(B_2^n,Z_q(\mu)) \ls 1 ~,~ P_F (Z_q(\mu)) \supseteq B_F ,
\end{equation}
which hold since $Z_q(\mu) \supseteq Z_2(\mu) = B_2^n$. 
\qed
\end{proof}

An elementary computation based on Corollary \ref{cor:M} then yields a non-trivial estimate for $M(Z_q(\mu ))$. It is interesting to note that without using the trivial information that $Z_q(\mu) \supseteq B_2^n$ (or equivalently, the trivial estimates in (\ref{eq:trivial})), Corollary \ref{cor:M} would not yield anything meaningful.

\begin{theorem}\label{thm:M-Zq}
For any isotropic log-concave probability measure $\mu$ on $\Real^n$:
\begin{equation*}
2\ls q\ls q_0 := (n\log (e+n))^{2/5} \;\;\; \Longrightarrow \;\;\;  M(Z_q(\mu)) \ls C\frac{\sqrt{\log q}}{\sqrt[4]{q}} .
\end{equation*}
\end{theorem}

\begin{proof}
We use the estimate:
\[ \sqrt{n}M(Z_q(\mu ))\ls C\sum_{k=1}^n\frac{1}{\sqrt{k}}\min\set{ 1, C\frac{1}{\min(\sqrt{q} , \sqrt[4]{k})} \frac{n}{k} \log\left( e+ \frac{n}{k}\right) } ,
\]
which follows from Corollary \ref{cor:M} combined with Theorem \ref{thm:Rqk-variant}.
We set $k(n,q)=(n\log q)/\sqrt{q}$. Note that if $k\gr k(n,q)$ then $k\gr c q^2$. Therefore,
we may write:
\begin{align*}
\sqrt{n}M(Z_q(\mu )) &\ls C\sum_{k=1}^{k(n,q)}\frac{1}{\sqrt{k}}+\frac{Cn}{\sqrt{q}}\sum_{k=k(n,q)}^n\frac{1}{k^{3/2}}\log\left( e+ \frac{n}{k}\right)\\
&\ls C_1\sqrt{k(n,q)}+C_2\frac{n\log q}{\sqrt{q k(n,q)}}\ls C_3\frac{\sqrt{n\log q}}{\sqrt[4]{q}}.
\end{align*}
The result follows. 
\qed
\end{proof}

For larger values of $q$, we obtain no additional information beyond the trivial monotonicity:
\[
 q_0 \ls q \;\;\; \Longrightarrow \;\;\;  M(Z_q(\mu)) \ls M(Z_{q_0}(\mu)) \ls C \frac{\log^{2/5} (e+n)}{n^{1/10}}.
\]

If $K$ is an isotropic centrally-symmetric convex body in $\Real^n$, using that $\lambda_{K/L_K}$ is isotropic log-concave and that $Z_n(\lambda_{K/L_K})$ is isomorphic to $K/L_K$, one immediately translates the above results to corresponding estimates for $K$.

\begin{theorem}
Given $k = 1,\ldots,n$, set:
\[
R_{k} := \min \set{ 1, C \frac{1}{\sqrt[4]{k}} \frac{n}{k} \log\left( e+ \frac{n}{k}\right) } .
\]
Then, for any isotropic centrally-symmetric convex body $K$ in $\Real^n$:
\[
e_k(B_2^n,K) \ls \frac{R_{k}}{L_K} ,
\]
and there exists $F \in G_{n,n-k}$ so that:
\[
P_F(K) \supseteq \frac{L_K}{R_{k}} B_F .
\]
Moreover:
\[
M(K)\ls \frac{C}{L_K} \frac{\log^{2/5} (e+n)}{n^{1/10}} .
\]
\end{theorem}

\subsection{Assuming that the isotropic constant is bounded}

It is interesting to perform the same calculations under the assumption
that $L_\mu \ls C$ for any log-concave probability measure $\mu$ (regardless of dimension). In that case:
\[
v_k^{-}(Z_q(\mu)) \gr c \sqrt{\min(q, k)}.
\]
This would yield the following conditional result:

\begin{theorem}
Given $q \gr 2$ and an integer $k = 1,\ldots,n$, denote:
\[
R_{k,q} := \min \set{ 1, C \frac{1}{\sqrt{\min(q,k)}} \frac{n}{k} \log\left( e+ \frac{n}{k}\right) } .
\]
Assuming that $L_\mu \ls C$ for any log-concave probability measure {\rm (}regardless of dimension{\rm )},
then for any isotropic log-concave probability measure $\mu$ on $\Real^n$:
\[
e_k(B_2^n,Z_q(\mu)) \ls R_{k,q} ,
\]
and there exists $F \in G_{n,n-k}$ so that:
\[
P_F (Z_q(\mu)) \supseteq \frac{1}{R_{k,q}} B_F .
\]
Furthermore:
\[
M(Z_q(\mu)) \ls C\frac{\sqrt{\log q}}{\sqrt[4]{q}}\;\;\hbox{for all}\;\; 2\ls q\ls (n\log n)^{2/3} .
\]
Consequently, for every isotropic convex body $K$ in ${\mathbb R}^n$ one would have:
\[
M(K)\ls C \frac{\log^{1/3} (e+ n)}{n^{1/6}} .
\]
\end{theorem}

\subsection{$\psi_\alpha$--measures}

Finally, rather than assuming that $L_\mu$ is always bounded, we repeat the calculations for a log-concave measure $\mu$ which is assumed to be $\psi_\alpha$-regular. Recall that $\mu$ is called $\psi_\alpha$ with constant $b_\alpha$ ($\alpha \in [1,2]$) if:
\[
 Z_q(\mu) \subseteq b_\alpha q^{1/\alpha} Z_2(\mu)\;\;\; \hbox{for all}\; q \gr 2 .
\]
Note that this property is inherited by all marginals of $\mu$, and that any log-concave measure is $\psi_1$ with $b_1 = C$ a universal constant.

It was shown by Klartag and E. Milman \cite{Klartag-EMilman-2012} that when $\mu$ is a $\psi_\alpha$ log-concave probability measure on $\Real^n$ with constant $b_\alpha$, then:
\[
1 \ls q \ls C \frac{n^{\frac{\alpha}{2}}}{b_\alpha^\alpha} \;\;\; \Longrightarrow \;\;\;  \vrad (Z_q(\mu )) \gr c \sqrt{q} ,
\]
and:
\[
L_{\mu} \ls C \sqrt{ b_\alpha^\alpha n^{1 - \alpha/2}} .
\]
This implies that for such a measure, for any $H \in G_{n,k}$:
\[
1 \ls q \ls C \frac{k^{\frac{\alpha}{2}}}{b_\alpha^\alpha} \;\;\; \Longrightarrow \;\;\;  \vrad (P_H(Z_q(\mu ))) \gr c \sqrt{q} .
\]
By \eqref{eq:3}, we know that:
\begin{equation} \label{eq:LYZ}
q \ge q_0 := C \frac{k^{\frac{\alpha}{2}}}{b_\alpha^\alpha} \;\;\; \Longrightarrow \;\;\;  \vrad (P_H (Z_q(\mu ))) \gr c' \max \brac{\sqrt{q_0},\frac{\sqrt{\min(q,k)}}{L_{\pi_H \mu}}} .
\end{equation}
Unfortunately, since we only know that:
\[
 L_{\pi_H \mu} \ls C \sqrt{ b_\alpha^\alpha k^{1-\alpha/2}} ,
\]
we again see that the maximum in (\ref{eq:LYZ}) is always attained by the $\sqrt{q_0}$ term. Summarizing, we have:

\begin{lemma}
Let $\mu $ be an isotropic log-concave probability measure on ${\mathbb R}^n$ which
is $\psi_\alpha$ with constant $b_\alpha$ for some $\alpha \in [1,2]$. Then for any $q\gr 1$ and $k=1,\ldots,n$ we have:
\[
v_k^{-}(Z_q(\mu)) \gr c \sqrt{\min \brac{q,\frac{k^{\alpha/2}}{b_\alpha^\alpha}}} .
\]
\end{lemma}

Plugging this estimate into the general results of Sections \ref{sec:covering} and \ref{sec:diam}, we obtain:
\begin{theorem}
Let $\mu$ denote an isotropic log-concave probability measure on $\Real^n$ which is $\psi_\alpha$ with constant $b_\alpha$
for some $\alpha \in [1,2]$. Given $q \gr 2$ and an integer $k = 1,\ldots,n$, denote:
\[
R_{k,q} := \min \set{ 1, C \frac{1}{\sqrt{\min \brac{q,\frac{k^{\alpha/2}}{b_\alpha^\alpha}}}} \frac{n}{k} \log\left( e+ \frac{n}{k}\right) } .
\]
Then:
\[
e_k(B_2^n,Z_q(\mu)) \ls R_{k,q} ,
\]
and there exists $F \in G_{n,n-k}$ so that:
\[
P_F (Z_q(\mu)) \supseteq \frac{1}{R_{k,q}} B_F .
\]
Furthermore:
\[
M(Z_q(\mu)) \ls C\frac{\sqrt{\log q}}{\sqrt[4]{q}}\;\;\hbox{for all}\;\; 2\ls q\ls c \frac{(n \log (e+n))^{\frac{2\alpha }{\alpha +4}}}{b_\alpha^\frac{4 \alpha}{\alpha+4}} .
\]
Consequently, for every isotropic convex body $K$ in ${\mathbb R}^n$ so that $\lambda_K$ is $\psi_{\alpha }$ with constant $b_{\alpha }$, one has:
\[
M(K)\ls \frac{C}{L_K} b_\alpha^{\frac{\alpha}{\alpha+4}} \frac{\log^{\frac{2}{\alpha +4}} (e+ n)}{n^{\frac{\alpha }{2(\alpha +4})}} .
\]
\end{theorem}

\begin{remark} %\rm
Better estimates for the entropy-numbers $e_k(B_2^n,Z_q(\mu))$ and Gelfand numbers $c_k(Z_q(\mu)^\circ)$ may be obtained for various ranges of $k$ by employing the alternative known estimates mentioned in Remark \ref{rem:alternative}. However, these do not result in improved estimates on $M(Z_q(\mu))$, which was our ultimate goal. We therefore leave these improved estimates on the entropy and Gelfand numbers to the interested reader. We only remark that even the classical low-$M^*$ estimate (\ref{eq:low-M*}) coupled with our estimate on $M(Z_q(\mu))$ yield improved estimates for $e_k$ and $c_k$ in a certain range - a type of ``bootstrap" phenomenon.
\end{remark}

\section{Concluding remarks}

In this section we briefly describe an improved and simplified version of
the arguments from \cite{Giannopoulos-Stavrakakis-Tsolomitis-Vritsiou-TAMS}
and compare the resulting improved estimates to the ones from the previous section.
Following the general approach we employ in this work, the arguments are presented
for general centrally--symmetric convex bodies, and this in fact further simplifies
the exposition of \cite{Giannopoulos-Stavrakakis-Tsolomitis-Vritsiou-TAMS}.

We mainly concentrate on presenting an alternative proof of the following
slightly weaker variant of Theorem \ref{thm:Rqk-variant}:

\begin{theorem}\label{th:projZq-main}
Let $K$ be a centrally-symmetric convex body in $\Real^n$. For any $k =1 , \ldots, \lfloor n/2 \rfloor$
there exists $F\in G_{n,n-2k}$ such that:
\begin{equation*}P_F\bigl(K\bigr)\supseteq
\frac{c }{\frac{n}{k}\log^2 \left (e+\frac{n}{k}\right)} v_k^{-}(K) \,B_F\end{equation*}
where $c>0$ is an absolute constant.
\end{theorem}

For the proof of Theorem \ref{th:projZq-main}, we use a sort of converse to Carl's theorem (\ref{eq:carl}) on the diameter of sections of a
convex body satisfying $2$-regular entropy estimates, which is due to V. Milman \cite{VMilman-1990} (see also \cite[Chapter 9]{BGVV-book-isotropic}).

\begin{lemma}\label{lem:milman-m*-covering}
Let $L$ be a symmetric convex body in $\mathbb R^n$. Then:
\[
\sqrt{k} \; c_k(L,B_2^n) \ls C \log(e + n/k) \sup_{k \ls m \ls n} \sqrt{m} \; e_m(L,B_2^n) .
\]
\end{lemma}

\begin{remark}
\rm Clearly, by applying a linear transformation, the statement equally holds with $B_2^n$ replaced by an arbitrary ellipsoid.
\end{remark}

%\begin{proof}
\noindent {\it Proof of Theorem \ref{th:projZq-main}}. 
Given $k=1,\ldots,\lfloor n/2 \rfloor $, let $\Eps = \Eps_{K,\alpha_k}$ denote Pisier's $\alpha_k$-regular $M$-ellipsoid, for some $\alpha_k \in [1,2)$ to be determined.
Instead of directly using Pisier's estimate (\ref{eq:Pisier-ck}) on the Gelfand numbers as in the proof of Theorem \ref{thm:main-s} to deduce the existence of $E \in G_{n,n-k}$ so that:
\begin{equation} \label{eq:compare1}
P_E(K) \supseteq \frac{1}{P_{\alpha_k}} \brac{\frac{k}{n}}^{1/\alpha_k} P_E (\Eps ) ,
\end{equation}
the starting point in \cite{Giannopoulos-Stavrakakis-Tsolomitis-Vritsiou-TAMS} are the more traditional covering estimates (\ref{eq:Pisier-ek}):
\begin{equation} \label{eq:compare2}
\max\{e_k(K,\Eps),e_k(K^\circ,\Eps^\circ),e_k(\Eps,K),e_k(\Eps^\circ,K^\circ)\} \ls P_\alpha \brac{\frac{n}{k}}^{1/\alpha_k} .
\end{equation}
In \cite{Giannopoulos-Stavrakakis-Tsolomitis-Vritsiou-TAMS}, the following estimate was used (see \cite[Theorem 5.14]{Pisier-book}):
\[
c_k(K^\circ,\Eps^\circ) \ls C \sqrt{\frac{n}{k}} e_k(K^\circ,\Eps^\circ) .
\]
However, this estimate does not take into account the regularity of the covering. Consequently, a significantly improved estimate is obtained by employing Lemma \ref{lem:milman-m*-covering} (and the subsequent remark) which exploits this regularity:
\begin{eqnarray*}
\sqrt{k} \; c_k(K^\circ,\Eps^\circ) & \ls & C \log(e + n/k) \sup_{k \ls m \ls n} \sqrt{m} e_m(K^\circ,\Eps^\circ) \\
& \ls &  C \log(e + n/k) \sup_{k \ls m \ls n} \sqrt{m}  \; P_{\alpha_k} \brac{\frac{n}{m}}^{1/\alpha_k} .
\end{eqnarray*}
Even with this improvement, note that this is where the current approach incurs some unnecessary logarithmic price with respect to the approach in the previous sections: instead of using (\ref{eq:compare1}) directly, one uses (\ref{eq:compare2}) which Pisier obtains from (\ref{eq:compare1}) by applying Carl's theorem, and then uses the converse to Carl's theorem (Lemma \ref{lem:milman-m*-covering}) to pass back to Gelfand number estimates.

Using $\alpha_k = 2 - \frac{1}{\log(e + n/k)}$, we deduce that:
\[
c_k(K^\circ,\Eps^\circ) \ls C  \sqrt{\frac{n}{k}} \log^{3/2}(e + n/k) ,
\]
or in other words, the existence of $E \in G_{n,n-k}$ such that:
\[
P_E(K) \supseteq \frac{1}{C  \sqrt{\frac{n}{k}} \log^{3/2}(e + n/k)} P_E (\Eps ) .
\]

The rest of the proof is identical to that of Theorem \ref{thm:main-s}. For the ellipsoid $\Eps' := P_E(\Eps)$ we may always find a linear subspace $F \subseteq E$ of codimension $k$ in $E$ so that:
\[
P_F (\Eps') \supseteq \inf_{H \in G_k(E)} \set{ \vrad(P_{H} (\Eps'))} B_F .
\]
Estimating $\vrad(P_H (\Eps')) = \vrad(P_H (\Eps))$ by comparing to $\vrad(P_H (K))$
via the dual covering estimate on $e_k(K,\Eps)$ (note that there is no need to use the duality of entropy theorem here), we obtain:
\[
\vrad(P_H(\Eps')) \gr \frac{1}{2 e_k(K,\Eps)} \vrad(P_H (K)) \gr \frac{1}{2 C  \sqrt{\frac{n}{k}} \log^{1/2}(e + n/k)} \vrad(P_H (K))  .
\]
Combining all of the above, we deduce the existence of $F \in G_{n,n-2k}$ so that:
\[
P_F(K) \supseteq \frac{1}{C' \frac{n}{k} \log^{2}(e + n/k)} \vrad(P_H (K)) B_F .
\]
This concludes the proof.
\qed
%\end{proof}

\medskip

Having obtained a rather regular estimate on the Gelfand numbers, the next goal is to obtain an entropy estimate. To this end,
 one can use Carl's theorem (\ref{eq:carl}) or (\ref{eq:carl2}), as we do in Section \ref{sec:diam}. The approach in \cite{Giannopoulos-Stavrakakis-Tsolomitis-Vritsiou-TAMS} proceeds by employing an entropy extension theorem of
Litvak, V. Milman, Pajor and Tomczak-Jaegermann \cite{Litvak-VMilman-Pajor-Tomczak-2006}. We remark that this too
may be avoided, by employing the following elementary covering estimate (see e.g. \cite[Chapter 9]{BGVV-book-isotropic}):

\begin{lemma}\label{lem:6.5.3}
Let $K$ be a symmetric convex body in $\mathbb R^n$ and assume that
$B_2^n\subseteq \rho K$ for some $\rho \gr 1$. Let $W$ be a subspace
of $\mathbb R^n$ with $\dim W=m$ and $P_{W^\perp}(K)\supseteq B_{W^\perp}$. Then, we have
\begin{equation*}\label{eq:6.5.10}
N(B_2^n, 4K)\ls \left(3 \rho\right)^m.
\end{equation*}
\end{lemma}

Finally, having a covering estimate at hand, the estimate on $M(K)$ is obtained by Dudley's entropy bound (\ref{eq:Dudley}). Plugging in the lower bounds on $v_k^{-}(Z_q(\mu))$ given in Section \ref{sec:results}, the results of \cite{Giannopoulos-Stavrakakis-Tsolomitis-Vritsiou-TAMS} are recovered and improved.

As the reader may wish to check, the improved approach of this section over the arguments of \cite{Giannopoulos-Stavrakakis-Tsolomitis-Vritsiou-TAMS} yields estimates which are almost as good as the ones obtained in Section \ref{sec:results}, and only lose by logarithmic terms.

\begin{acknowledgement}
The first named author acknowledges support from the programme ``API$\Sigma$TEIA II"
of the General Secretariat for Research and Technology of Greece. The second named author is supported by ISF (grant no. 900/10), BSF (grant no. 2010288), Marie-Curie Actions (grant no. PCIG10-GA-2011-304066) and the E. and J. Bishop Research Fund.
\end{acknowledgement}

\end{document}